\documentclass{amsart}
\usepackage{amsmath,amssymb,mathtools}
\usepackage[hidelinks]{hyperref}
\usepackage{microtype}

\allowdisplaybreaks
\numberwithin{equation}{section}

\newtheorem{theorem}{Theorem}[section]
\newtheorem{proposition}[theorem]{Proposition}
\newtheorem{lemma}[theorem]{Lemma}
\newtheorem{corollary}[theorem]{Corollary}
\theoremstyle{remark}
\newtheorem{remark}[theorem]{Remark}

\newcommand{\Z}{\mathbb Z}
\newcommand{\R}{\mathbb R}
\newcommand{\C}{\mathbb C}
\newcommand{\E}{\mathcal E}

\newcommand{\dd}{\,\mathrm d}

\begin{document}

\title[THE SHARP DISCRETE HARDY INEQUALITY ON $\Z^3$]
{The sharp discrete Hardy inequality on $\Z^3$}

\author{Natanael Alpay}
\address{Department of Mathematics, University of California, Irvine, Irvine, CA 92697, USA}
\email{nalpay@uci.edu}

\subjclass[2020]{39A12, 26D15, 47B39}
\keywords{Discrete Hardy inequality, lattice Laplacian, sharp constant, Hardy weight, completion of squares}

\date{}
\dedicatory{}
\commby{}

\begin{abstract}
We determine the sharp constant in the nearest-neighbor Hardy inequality on
$\Z^3$ with the Euclidean inverse-square weight.  For every finitely supported
function $u:\Z^3\to\C$, we prove
\[
\sum_{x\in\Z^3}\sum_{j=1}^3
 |u(x+e_j)-u(x)|^2
\geq
\frac14\sum_{x\in\Z^3\setminus\{0\}}
 \frac{|u(x)|^2}{|x|^2}.
\]
The coefficient $1/4$ is sharp, and equality is not attained by a nonzero finitely supported function.
The proof uses an explicit reciprocal edge field and an edgewise completion of squares, together with a concavity argument for the associated vertex weight.
\end{abstract}

\maketitle

\section{Introduction}

The classical Hardy inequality on $\R^d$ states that
\begin{equation}
\int_{\R^d}|\nabla f(x)|^2\dd x
\geq
\frac{(d-2)^2}{4}
\int_{\R^d}\frac{|f(x)|^2}{|x|^2}\dd x,
\qquad d\geq3,
\label{eq:continuous-hardy}
\end{equation}
for $f\in C_c^\infty(\R^d\setminus\{0\})$.  The coefficient
$(d-2)^2/4$ is sharp and is suggested by the formal ground state
$|x|^{-(d-2)/2}$.  This function is not an extremizer, but suitable
truncations form a minimizing sequence.

We study the nearest-neighbor analogue of \eqref{eq:continuous-hardy}.  Let
$e_1,\ldots,e_d$ be the standard basis of $\Z^d$.  For a finitely supported
function $u:\Z^d\to\C$, set
\begin{equation}
\E_d(u)
=
\sum_{x\in\Z^d}\sum_{j=1}^d
 |u(x+e_j)-u(x)|^2.
\label{eq:energy-d}
\end{equation}
Every undirected nearest-neighbor edge is counted once in
\eqref{eq:energy-d}. 
Equivalently, if the positive lattice Laplacian is defined by
\[
(\mathcal L_d u)(x)
=
\sum_{j=1}^d
\bigl(2u(x)-u(x+e_j)-u(x-e_j)\bigr),
\]
then
\[
\E_d(u)
=
\langle u,\mathcal L_d u\rangle_{\ell^2(\Z^d)}.
\]
Thus $\E_d$ is the Dirichlet form of the standard positive lattice
Laplacian.
Let
\[
\mathcal D_0(\Z^d)
=
\{u\in C_c(\Z^d;\C):u(0)=0\}.
\]
With the Euclidean norm, define
\begin{equation}
\lambda_d
=
\inf_{\substack{u\in\mathcal D_0(\Z^d)\\u\neq0}}
\frac{\displaystyle \E_d(u)}
{\displaystyle \sum_{x\in\Z^d\setminus\{0\}}\frac{|u(x)|^2}{|x|^2}}.
\label{eq:lambda-d}
\end{equation}
Thus $\lambda_d$ is the largest coefficient for which
\begin{equation}
\E_d(u)
\geq
\lambda_d
\sum_{x\neq0}\frac{|u(x)|^2}{|x|^2}
\label{eq:discrete-hardy-d}
\end{equation}
holds for every $u\in\mathcal D_0(\Z^d)$.  The result below is proved in the
stronger form valid for every finitely supported function, without imposing
$u(0)=0$.\\

Hardy's original inequality was discrete and one-dimensional
\cite{Hardy1920}; see \cite{KufnerMaligrandaPersson2006} for its history.
Multidimensional lattice inequalities arose in the spectral theory of
discrete Schr\"odinger operators; see Rozenblum and Solomyak
\cite{RozenblumSolomyak2009} and Kapitanski and Laptev
\cite{KapitanskiLaptev2016}.  A related problem is to construct an optimal
spatially varying Hardy weight rather than the best constant multiplying the
fixed weight $|x|^{-2}$.  Keller, Pinchover and Pogorzelski developed a
criticality theory for optimal Hardy weights on graphs
\cite{KellerPinchoverPogorzelski2018}.  For the free Laplacian on $\Z^d$, they
constructed an optimal Hardy weight $W_{\mathrm{KPP}}$ satisfying
\begin{equation}
W_{\mathrm{KPP}}(x)
=
\frac{(d-2)^2}{4|x|^2}
+
O(|x|^{-3}),
\qquad |x|\to\infty;
\label{eq:KPP-asymptotic}
\end{equation}
see \cite[Theorem~7.2]{KellerPinchoverPogorzelski2018}.  Keller and Lemm
subsequently studied the robustness of inverse-square Hardy-weight behavior
for more general elliptic coefficient fields on $\Z^d$
\cite{KellerLemm2023}.  The free-lattice asymptotic
\eqref{eq:KPP-asymptotic} does not determine the largest global coefficient
in \eqref{eq:discrete-hardy-d}: it gives no information about the finitely
many sites near the origin and does not determine the sign of the lower-order
remainder.\\

Gupta proved that the sharp Hardy coefficient is of order $d$ as
the dimension tends to infinity, in contrast to the quadratic growth of the
continuous coefficient; equivalently, $\lambda_d\asymp d$ as $d\to\infty$
\cite{Gupta2023}.  More recently, Gupta \cite{Gupta2026} and, independently,
Huang and Ye \cite{HuangYe2026} obtained the precise asymptotic formula
\[
\lim_{d\to\infty}\frac{\lambda_d}{d}=2.
\]
These results do not give the exact coefficient in a fixed low dimension.
Weighted, nonlinear, and improved one-dimensional forms have been studied in
\cite{FischerKellerPogorzelski2023,Dyda2025,YessirkegenovZhangirbayev2026}.
Recent accounts still list the explicit multidimensional pure-power constant
as an open problem \cite{YessirkegenovZhangirbayev2026}.

\section{Main results}
We focus on $d=3$, the lowest dimension covered by the classical
inequality \eqref{eq:continuous-hardy}.  In this dimension the continuous
Hardy coefficient is $1/4$. Our main result is that the same coefficient is sharp on the lattice,
so that $\lambda_3=1/4$.

For $n\geq1$, define the quantity
\begin{equation}
\rho(n)
=
\frac{928n^2+5356n-2197}
{4n(8n+13)(64n^2+108n+169)}.
\label{eq:rho-definition}
\end{equation}

\begin{theorem}[Improved Hardy inequality]
\label{thm:main}
For every finitely supported function $u:\Z^3\to\C$,
\begin{equation}
\begin{aligned}
\E_3(u)-\frac14\sum_{x\neq0}\frac{|u(x)|^2}{|x|^2}
&\geq
\frac{12}{13}|u(0)|^2
+
\sum_{x\neq0}\rho(|x|^2)|u(x)|^2.
\end{aligned}
\label{eq:explicit-remainder}
\end{equation}
The right-hand side is strictly positive whenever $u\not\equiv0$.
\end{theorem}

Since the right-hand side of \eqref{eq:explicit-remainder} is nonnegative,
Theorem~\ref{thm:main} gives the Hardy inequality with coefficient $1/4$.
Proposition~\ref{prop:sharpness} later shows that this coefficient cannot
be improved.

\begin{corollary}[Sharp Hardy inequality]
\label{cor:sharp-constant}
For every finitely supported function $u:\Z^3\to\C$,
\begin{equation}
\E_3(u)= \sum_{x\in\Z^3}\sum_{j=1}^3
 |u(x+e_j)-u(x)|^2
\geq
\frac14
\sum_{x\in\Z^3\setminus\{0\}}
 \frac{|u(x)|^2}{|x|^2}.
\label{eq:main-hardy}
\end{equation}
The coefficient $1/4$ is sharp, and the inequality is strict whenever
$u\not\equiv0$.  In particular,
\[
\lambda_3=\frac14.
\]
\end{corollary}

The lower bound in Theorem~\ref{thm:main} follows from an explicit edge
field.  Write $x\sim y$ when $x$ and $y$ are nearest neighbors.  For an
ordered neighboring pair, define
\begin{equation}
a_{xy}
=
\frac{5|x|^2+3|y|^2+8}
{3|x|^2+5|y|^2+8}.
\label{eq:a-xy}
\end{equation}
Then $a_{xy}>0$ and
\begin{equation}
a_{yx}=a_{xy}^{-1}.
\label{eq:reciprocity}
\end{equation}

The choice of $a_{xy}$ is motivated by the formal ground state $h(x)=|x|^{-1/2}$ of the
critical three-dimensional Hardy operator,
\[
\left(-\Delta-\frac{1}{4|x|^2}\right)h=0
\qquad\text{on }\R^3\setminus\{0\}.
\]
The function $h$ is not an admissible extremizer for the Hardy inequality,
but it is the positive zero-energy solution associated with the sharp Hardy
potential; see, for example,
\cite{FrankSeiringer2008,PinchoverTintarev2006}.  For a nearest-neighbor
edge $x\sim y$ with $x,y\neq0$, write
\[
n=|x|^2,\qquad m=|y|^2,\qquad \delta=m-n.
\]
Then
\[
\frac{h(y)}{h(x)}
=
\left(1+\frac{\delta}{n}\right)^{-1/4}
=
1-\frac{\delta}{4n}
+
O\left(\frac{\delta^2}{n^2}\right),
\qquad |x|\to\infty,
\]
and $a_{xy}$ has the same first-order behavior.  
The construction of the edge field is discussed
further in Appendix~\ref{appendix:a_xy}.
Associated with this edge field, define the weight
\begin{equation}
W(x)
=
\sum_{y\sim x}(1-a_{xy}).
\label{eq:W-definition}
\end{equation}

Individual terms in this sum need not be positive.  The point is that,
after the six neighboring contributions are combined, $W(x)$ is positive and satisfies
the lower bound in Proposition~\ref{prop:weight-bound}.

Let
\begin{equation}
\label{eq:Eplus}
\mathcal{E}_+
=
\{(x,x+e_j):x\in\Z^3,\ 1\leq j\leq3\}
\end{equation}
be the set of forward oriented edges.  Every undirected nearest-neighbor edge
occurs exactly once in $\mathcal E_+$, in agreement with the normalization in
\eqref{eq:energy-d}.

\begin{proposition}[Exact edge decomposition]
\label{prop:edge-decomposition}
For every finitely supported $u:\Z^3\to\C$,
\begin{equation}
\begin{aligned}
&\E_3(u)-\sum_{x\in\Z^3}W(x)|u(x)|^2=
\sum_{(x,y)\in\mathcal{E}_+}
\left|
 a_{xy}^{1/2}u(x)-a_{xy}^{-1/2}u(y)
\right|^2.
\end{aligned}
\label{eq:edge-decomposition}
\end{equation}
In particular,
\[
\E_3(u)\geq\sum_{x\in\Z^3}W(x)|u(x)|^2.
\]
\end{proposition}

This proposition is key in the proof.  The next step is to obtain a lower
bound for $W(x)$.

\begin{proposition}[Lower bound for $W(x)$]
\label{prop:weight-bound}
The weight in \eqref{eq:W-definition} satisfies
\[
W(0)=\frac{12}{13}.
\]
Moreover, for every $x\neq0$, and $\rho$ as in \eqref{eq:rho-definition}, we have
\begin{equation}
W(x)
\geq
\frac{1}{4|x|^2}
+
\rho(|x|^2).
\label{eq:W-rho-bound}
\end{equation}
Equality in \eqref{eq:W-rho-bound} holds if and only if $x$ lies on one of the three coordinate axes.
In particular, since $\rho(n)>0$ for every integer $n\geq1$,
\begin{equation}
\label{eq:W-bound-squared}
    W(x)>\frac{1}{4|x|^2},
\qquad x\neq0.
\end{equation}
\end{proposition}

The following proposition gives the sharpness needed in
Corollary~\ref{cor:sharp-constant}.

\begin{proposition}[Sharpness]
\label{prop:sharpness}
There exists a sequence of nonzero functions $u_k\in C_c(\Z^3)$ satisfying
$u_k(0)=0$ and
\[
\lim_{k\to\infty}
\frac{\E_3(u_k)}
{\displaystyle\sum_{x\neq0}\frac{|u_k(x)|^2}{|x|^2}}
=
\frac14.
\]
\end{proposition}

Theorem~\ref{thm:main} and Corollary~\ref{cor:sharp-constant} are proved in
Section~\ref{sec:main-proof} after the proofs of these propositions.

\begin{remark}[Optimal weights]\label{rmk:optweight}
\textit{The weight $W$ is not claimed to be a critical or pointwise optimal Hardy
weight.  Our result concerns the best constant multiplying the fixed weight
$|x|^{-2}$.}
\end{remark}

\subsection*{Acknowledgement}
The author thanks Professor Paata Ivanisvili for his guidance and continued support.

\section{Proof of Proposition~\ref{prop:edge-decomposition}}
\label{sec:edge-proof}

We start with a two-point identity.

\begin{lemma}
\label{lem:edge-square}
Let $a>0$ and $z,w\in\C$.  Then
\begin{equation}
\begin{aligned}
|z-w|^2
&=(1-a)|z|^2+(1-a^{-1})|w|^2
+
\left|a^{1/2}z-a^{-1/2}w\right|^2.
\end{aligned}
\label{eq:two-point-square}
\end{equation}
\end{lemma}

\begin{proof}
Expanding the last square gives
\[
a|z|^2+a^{-1}|w|^2-2\operatorname{Re}(z\overline w).
\]
Adding the first two terms on the right-hand side of
\eqref{eq:two-point-square}, the coefficient of $|z|^2$ is
$(1-a)+a=1$, and the coefficient of $|w|^2$ is
$(1-a^{-1})+a^{-1}=1$.  The resulting expression is
\[
|z|^2+|w|^2-2\operatorname{Re}(z\overline w)
=|z-w|^2.
\]
\end{proof}

\begin{proof}[Proof of Proposition~\ref{prop:edge-decomposition}]
For an ordered neighboring pair $x\sim y$, the numerator and denominator in
\eqref{eq:a-xy} are positive.  Interchanging $x$ and $y$ gives
\[
a_{yx}
=
\frac{5|y|^2+3|x|^2+8}
{3|y|^2+5|x|^2+8}
=
a_{xy}^{-1}.
\]

Next, we apply Lemma~\ref{lem:edge-square} to every forward edge
$(x,y)\in\mathcal E_+$, with $z=u(x)$, $w=u(y)$, and $a=a_{xy}$.
Using the reciprocal relation
$
a_{yx}=a_{xy}^{-1},
$
we obtain
\begin{equation}\label{eq:red}
    \begin{aligned}
|u(x)-u(y)|^2
={}&
(1-a_{xy})|u(x)|^2
+
(1-a_{yx})|u(y)|^2
\\
&+
\left|
a_{xy}^{1/2}u(x)
-
a_{xy}^{-1/2}u(y)
\right|^2.
\end{aligned}
\end{equation}
Substituting \eqref{eq:red} into \eqref{eq:energy-d}  and summing over all $(x,y)\in\mathcal E_+$ gives
\[
\begin{aligned}
\mathcal E_3(u)
=&
\underbrace{\sum_{(x,y)\in\mathcal E_+}
(1-a_{xy})|u(x)|^2}_{=S_1}
+
\underbrace{\sum_{(x,y)\in\mathcal E_+}
(1-a_{yx})|u(y)|^2}_{=S_2}
\\
&+
\sum_{(x,y)\in\mathcal E_+}
\left|
a_{xy}^{1/2}u(x)
-
a_{xy}^{-1/2}u(y)
\right|^2.
\end{aligned}
\]

The last sum is the nonnegative remainder in
\eqref{eq:edge-decomposition}. We now regroup the first two sums
according to vertices.

Fix a vertex $v\in\mathbb Z^3$.  The three forward edges starting at $v$
are
\[
(v,v+e_j),
\qquad j=1,2,3.
\]
By \eqref{eq:red}, the coefficient attached to the first endpoint of a
forward edge $(x,y)$ is $1-a_{xy}$.  Hence the first sum contributes
\[
S_1(v)
:=
\sum_{j=1}^3
(1-a_{v,v+e_j})|u(v)|^2
\]
to the coefficient of $|u(v)|^2$.\\

We now consider the second sum,
\[
\sum_{(x,y)\in\mathcal E_+}
(1-a_{yx})|u(y)|^2.
\]
Here $v$ appears as the second endpoint of the three forward edges
\[
(v-e_j,v),
\qquad j=1,2,3.
\]
For such an edge, we have $x=v-e_j$ and $y=v$.  Therefore, again by
\eqref{eq:red}, the corresponding contribution is
\[
(1-a_{v,v-e_j})|u(v)|^2.
\]
Thus the total contribution from the three forward edges ending at $v$ is
\[
S_2(v)
:=
\sum_{j=1}^3
(1-a_{v,v-e_j})|u(v)|^2.
\]

Adding the outgoing and incoming contributions gives
\[
\begin{aligned}
S_1(v)+S_2(v)
&=
\sum_{j=1}^3
\left[
(1-a_{v,v+e_j})
+
(1-a_{v,v-e_j})
\right]|u(v)|^2
\\
&=
\left(
\sum_{y\sim v}(1-a_{vy})
\right)|u(v)|^2
\\
&=
W(v)|u(v)|^2.
\end{aligned}
\]
Since this holds for every vertex $v\in\mathbb Z^3$, summing over $v$
gives
\[
\begin{aligned}
&\sum_{(x,y)\in\mathcal E_+}
(1-a_{xy})|u(x)|^2
+
\sum_{(x,y)\in\mathcal E_+}
(1-a_{yx})|u(y)|^2
=
\sum_{v\in\mathbb Z^3}
W(v)|u(v)|^2.
\end{aligned}
\]
Consequently,
\[
\begin{aligned}
\mathcal E_3(u)
={}&
\sum_{x\in\mathbb Z^3}W(x)|u(x)|^2
+
\sum_{(x,y)\in\mathcal E_+}
\left|
a_{xy}^{1/2}u(x)
-
a_{xy}^{-1/2}u(y)
\right|^2.
\end{aligned}
\]
Rearranging gives \eqref{eq:edge-decomposition}.  Since the right-hand
side is nonnegative,
\[
\mathcal E_3(u)
\geq
\sum_{x\in\Z^3}W(x)|u(x)|^2.
\]
This completes the proof.
\end{proof}

\section{Proof of Proposition~\ref{prop:weight-bound}}
\label{sec:weight-proof}

To prove Proposition~\ref{prop:weight-bound}, we find an explicit expression for $W(x)$ defined by \eqref{eq:W-definition}. We then prove concavity, which allows us to find the one variable lower bound 
$$W(x)>\frac{1}{4|x|^2}.$$

Fix $x=(x_1,x_2,x_3)\in\Z^3\setminus\{0\}$ and write
$$n=|x|^2 = x_1^2 + x_2^2+x_3^2 .$$
For $y\sim x$, formula \eqref{eq:a-xy} gives
\begin{equation}
1-a_{xy}
=
\frac{2(|y|^2-|x|^2)}
{3|x|^2+5|y|^2+8}.
\label{eq:one-minus-a}
\end{equation}
For the two neighbors in the $j$-th coordinate direction, taking the regular Euclidean norm yields
\[
|x\pm e_j|^2=n+1\pm2x_j.
\]
Set
\begin{equation}
A=8n+13.
\label{eq:A-definition}
\end{equation}
Taking $y=x+e_j$ with $|y|^2 - |x|^2 = 1+2x_j$, allows to rewrite the denominator of \eqref{eq:one-minus-a} as follows
\[
3n+5(n+1+2x_j) + 8 = 8n + 13+ 10x_j = A + 10 x_j.
\]
Then \eqref{eq:one-minus-a} becomes
\[
1-a_{x,x+e_j}
=
\frac{2(1+2x_j)}{A+10x_j}
\quad \text{and}\quad 
1-a_{x,x-e_j}
=
\frac{2(1-2x_j)}{A-10x_j}.
\]
Pairing these terms, we obtain
\begin{equation}
\begin{aligned}
&(1-a_{x,x+e_j})+(1-a_{x,x-e_j})
=
\frac{4(A-20x_j^2)}{A^2-100x_j^2}.
\end{aligned}
\label{eq:paired-weight}
\end{equation}
Consequently, substituting  \eqref{eq:paired-weight} in  \eqref{eq:W-definition}, we get
\begin{equation}
W(x)
=
\sum_{j=1}^3
\frac{4(A-20x_j^2)}{A^2-100x_j^2}.
\label{eq:explicit-W}
\end{equation}

Using the above explicit formula for $W(x)$, our next step is to find a lower bound. 
To this end, fix $n\geq1$ and define
\begin{equation}
f_n(s)
=
\frac{4(A-20s)}{A^2-100s},
\quad 0\leq s\leq n,\quad A=8n+13.
\label{eq:f-n}
\end{equation}
The denominator in \eqref{eq:f-n} is positive, since
\begin{equation}
A^2-100s
\geq
A^2-100n
=
64n^2+108n+169
>0\quad \forall n\geq 1.
\label{eq:positive-denominator}
\end{equation}
A direct differentiation gives
\begin{equation*}
f_n''(s)
=
\frac{16000A(5-A)}{(A^2-100s)^3}.
\label{eq:f-second}
\end{equation*}
Since $A=8n+13>5$, it follows that
\begin{equation}
f_n''(s)<0,
\qquad 0\leq s\leq n.
\label{eq:f-concave}
\end{equation}
Thus $f_n$ is strictly concave  on $[0,n]$.\\

Set $s_j=x_j^2$ for $j=1,2,3$.  Then $s_j\geq0$ and
\[
s_1+s_2+s_3=n.
\]
Since the coefficients $s_j/n$ are nonnegative and sum to one,
\begin{equation}
(s_1,s_2,s_3)
=
\frac{s_1}{n}(n,0,0)
+
\frac{s_2}{n}(0,n,0)
+
\frac{s_3}{n}(0,0,n).
\label{eq:simplex-combination}
\end{equation}
Thus $(s_1,s_2,s_3)$ is a convex combination of the three vertices of
the simplex
\[
\Sigma_n
=
\{(r_1,r_2,r_3):r_j\geq0,\ r_1+r_2+r_3=n\}.
\]

Define
\[
F_n(r_1,r_2,r_3)
=
f_n(r_1)+f_n(r_2)+f_n(r_3).
\]
Since $f_n$ is concave on $[0,n]$, the function $F_n$ is concave on
$[0,n]^3$, and hence on $\Sigma_n$.

The triples arising from lattice points form only a subset of $\Sigma_n$,
since they have the special form
\[
(x_1^2,x_2^2,x_3^2).
\]
We work on the larger continuous simplex; any lower bound obtained there
therefore applies in particular to the lattice triples.

Applying concavity of $F_n$ to \eqref{eq:simplex-combination} gives
\[
\begin{aligned}
F_n(s_1,s_2,s_3)
\geq{}&
\frac{s_1}{n}F_n(n,0,0)
+
\frac{s_2}{n}F_n(0,n,0)
+
\frac{s_3}{n}F_n(0,0,n).
\end{aligned}
\]
Since $F_n$ is symmetric in its three variables,
\[
F_n(n,0,0)
=
F_n(0,n,0)
=
F_n(0,0,n).
\]
Therefore,
\begin{equation}
\begin{aligned}
F_n(s_1,s_2,s_3)
&\geq
\left(
\frac{s_1+s_2+s_3}{n}
\right)F_n(n,0,0)
\\
&=
F_n(n,0,0),
\end{aligned}
\label{eq:Fn}
\end{equation}
where we used $s_1+s_2+s_3=n$ in the last equality.

Thus we have
\[
W(x)=\sum_{j=1}^3
\frac{4(A-20x_j^2)}{A^2-100x_j^2} 
= f_n(s_1)+f_n(s_2)+f_n(s_3) = F_n(s_1,s_2,s_3).
\]
By \eqref{eq:Fn} it follows
\begin{equation}
\begin{aligned}
W(x)
&\geq F_n(n,0,0)
=f_n(n)+2f_n(0).
\end{aligned}
\label{eq:axis-minimum}
\end{equation}
Using \eqref{eq:f-n} and $A=8n+13$, we have
\begin{equation}
\begin{aligned}
W(x)
\geq  f_n(n)+2f_n(0)
&=
\frac{4(A-20n)}{A^2-100n}+\frac8A
\\
&=
\frac{4(32n^2+164n+507)}
{(8n+13)(64n^2+108n+169)}.
\end{aligned}
\label{eq:axis-value}
\end{equation}

Subtracting $1/(4n)$ from the lower bound on the right-hand side of
\eqref{eq:axis-value}, we obtain
\begin{equation}
\begin{aligned}
&f_n(n)+2f_n(0)-\frac1{4n}
=
\frac{928n^2+5356n-2197}
{4n(8n+13)(64n^2+108n+169)}.
\end{aligned}
\label{eq:positive-rational-remainder}
\end{equation}
The denominator is positive. Moreover, since $n\geq1$ 
\[
928n^2+5356n-2197
\geq
928+5356-2197
=4087>0.
\]
By the substituting definition of $\rho$ in \eqref{eq:rho-definition} into \eqref{eq:positive-rational-remainder} , we have
\[
f_n(n)+2f_n(0)-\frac1{4n} = \rho(n)>0.
\]
Since $ n=|x|^2$, combining the above with \eqref{eq:axis-minimum}, and rearranging we get
\[
W(x)
\geq
\frac1{4|x|^2}
+
\rho(|x|^2),
\]
which proves \eqref{eq:W-rho-bound}. Since $\rho(n)>0$, \eqref{eq:W-bound-squared} follows.\\

It remains to evaluate the weight at the origin.  If $x=0$ and $y\sim0$,
then $|y|^2=1$,
\[
a_{0y}
=
\frac{11}{13}.
\]
There are six neighbors, and therefore
\[
W(0)
=6\left(1-\frac{11}{13}\right)
=\frac{12}{13}.
\]
This proves Proposition~\ref{prop:weight-bound}.

\begin{remark}[Individual edge contributions]\label{rmk:edgesum}
\textit{For illustration, take $x=e_1$.  Its neighbor $0$ gives
$
1-a_{e_1,0}
=
1-\frac{13}{11}
=
-\frac{2}{11},$
while the neighbor $2e_1$ gives
$
1-a_{e_1,2e_1}
=
1-\frac{25}{31}
=
\frac{6}{31}.
$
Each of the four remaining neighbors has squared norm $2$ and contributes
$
1-\frac{19}{21}
=
\frac{2}{21}.
$
Consequently,
\[
W(e_1)
=
-\frac{2}{11}
+\frac{6}{31}
+4\frac{2}{21}
=
\frac{2812}{7161}
>
\frac14.
\]
Thus the individual edge contributions in the definition of $W$ need not
be positive.  The lower bound arises only after the contributions from all
six lattice directions are combined.}
\end{remark}

\section{Proof of Proposition~\ref{prop:sharpness}}
\label{sec:sharpness-proof}

We first construct an explicit family of near-extremizers for the continuous
three-dimensional Hardy inequality.  Fix a nonzero real-valued function
\[
\phi\in C_c^\infty((0,1);\R).
\]
For $L>1$, define
\begin{equation}
\varphi_L(x)
=
|x|^{-1/2}
\phi\left(\frac{\log|x|}{L}\right),
\qquad x\in\R^3\setminus\{0\}.
\label{eq:phi-L}
\end{equation}
Since $\phi$ has compact support contained in $(0,1)$, the function
$\varphi_L$ is smooth, compactly supported in an annulus separated from the
origin, and vanishes in a neighborhood of the origin.  We extend it by
setting $\varphi_L(0)=0$; the extension belongs to $C_c^\infty(\R^3)$.

For a radial function of the form
\[
\varphi(r)=r^{-1/2}g(\log r),
\]
the change of variables $t=\log r$ gives
\begin{equation}
\int_{\R^3}\frac{|\varphi(x)|^2}{|x|^2}\dd x
=
4\pi\int_{\R}|g(t)|^2\dd t
\label{eq:radial-Hardy-mass}
\end{equation}
and
\begin{equation}
\int_{\R^3}|\nabla\varphi(x)|^2\dd x
=
4\pi\int_{\R}
\left|g'(t)-\frac12g(t)\right|^2\dd t.
\label{eq:radial-energy}
\end{equation}
For \eqref{eq:phi-L}, we have
\[
g(t)=\phi(t/L).
\]
The mixed term in \eqref{eq:radial-energy} vanishes because
\[
\operatorname{Re}\int_{\R}g'(t)\overline{g(t)}\dd t
=
\frac12\int_{\R}(|g(t)|^2)'\dd t
=0.
\]
After the change of variables $s=t/L$, formulas
\eqref{eq:radial-Hardy-mass} and \eqref{eq:radial-energy} give
\begin{equation}
\int_{\R^3}\frac{|\varphi_L(x)|^2}{|x|^2}\dd x
=
4\pi L\int_0^1|\phi(s)|^2\dd s
\label{eq:phi-L-mass}
\end{equation}
and
\begin{equation}
\int_{\R^3}|\nabla\varphi_L(x)|^2\dd x
=
4\pi L
\left(
\frac14\int_0^1|\phi(s)|^2\dd s
+
\frac1{L^2}\int_0^1|\phi'(s)|^2\dd s
\right).
\label{eq:phi-L-energy}
\end{equation}
Hence
\begin{equation}
\frac{\displaystyle\int_{\R^3}|\nabla\varphi_L|^2\dd x}
{\displaystyle\int_{\R^3}|\varphi_L|^2/|x|^2\dd x}
=
\frac14
+
\frac1{L^2}
\frac{\displaystyle\int_0^1|\phi'(s)|^2\dd s}
{\displaystyle\int_0^1|\phi(s)|^2\dd s}.
\label{eq:continuous-near-extremizer}
\end{equation}
The right-hand side tends to $1/4$ as $L\to\infty$.

We now pass from the continuous functions to the lattice.  For fixed $L$ and
$N\in\mathbb N$, set
\begin{equation}
u_{L,N}(x)
=
\varphi_L(x/N),
\qquad x\in\Z^3.
\label{eq:u-LN}
\end{equation}
The function $u_{L,N}$ is finitely supported and satisfies $u_{L,N}(0)=0$.
For each fixed $L$, there is a compact set $K_L$, depending on $L$ but not
on $N$, that contains the supports of all the difference quotients below.
For each coordinate direction,
\[
N\left(
\varphi_L(y+e_j/N)-\varphi_L(y)
\right)
\longrightarrow
\partial_j\varphi_L(y)
\]
uniformly on $K_L$.  Therefore
\begin{equation}
\begin{aligned}
\frac{\E_3(u_{L,N})}{N}
&=
\frac1{N^3}
\sum_{x\in\Z^3}\sum_{j=1}^3
\left|
N\left(
\varphi_L((x+e_j)/N)-\varphi_L(x/N)
\right)
\right|^2
\\
&\longrightarrow
\int_{\R^3}|\nabla\varphi_L(y)|^2\dd y.
\end{aligned}
\label{eq:energy-Riemann-sum}
\end{equation}
Likewise,
\begin{equation}
\begin{aligned}
\frac1N
\sum_{x\neq0}\frac{|u_{L,N}(x)|^2}{|x|^2}
&=
\frac1{N^3}
\sum_{x\neq0}
\frac{|\varphi_L(x/N)|^2}{|x/N|^2}
\longrightarrow
\int_{\R^3}\frac{|\varphi_L(y)|^2}{|y|^2}\dd y.
\end{aligned}
\label{eq:mass-Riemann-sum}
\end{equation}
There is no singularity in this Riemann sum because $\varphi_L$ vanishes in a
neighborhood of the origin. 
The limiting Hardy mass is strictly positive
because $\varphi_L\not\equiv0$.  Moreover, $u_{L,N}\not\equiv0$ for all
sufficiently large $N$: one may choose a point at which $\varphi_L$ is
nonzero and approximate it by points of $N^{-1}\Z^3$.  It follows from
\eqref{eq:energy-Riemann-sum} and \eqref{eq:mass-Riemann-sum} that, for each
fixed $L$,
\begin{equation}
\frac{\E_3(u_{L,N})}
{\displaystyle\sum_{x\neq0}|u_{L,N}(x)|^2/|x|^2}
\longrightarrow
\frac{\displaystyle\int_{\R^3}|\nabla\varphi_L|^2\dd y}
{\displaystyle\int_{\R^3}|\varphi_L(y)|^2/|y|^2\dd y}
\qquad(N\to\infty).
\label{eq:quotient-Riemann-limit}
\end{equation}

For each integer $L\geq2$, choose $N(L)$ sufficiently large that
$u_{L,N(L)}\not\equiv0$ and its lattice Rayleigh quotient differs from the
continuous quotient in \eqref{eq:continuous-near-extremizer} by at most
$L^{-2}$.  Then
\[
u_L:=u_{L,N(L)}
\]
is a sequence of nonzero finitely supported lattice functions with
$u_L(0)=0$, and
\[
\frac{\E_3(u_L)}
{\displaystyle\sum_{x\neq0}|u_L(x)|^2/|x|^2}
\longrightarrow
\frac14.
\]
This proves Proposition~\ref{prop:sharpness}.

\section{Proof of Theorem~\ref{thm:main} and Corollary~\ref{cor:sharp-constant}}
\label{sec:main-proof}

\begin{proof}[Proof of Theorem~\ref{thm:main}]
Combining the exact decomposition \eqref{eq:edge-decomposition} with the
Hardy term gives
\begin{equation}
\begin{aligned}
&\mathcal E_3(u)
-
\frac14\sum_{x\neq0}\frac{|u(x)|^2}{|x|^2}
\\
&\qquad=
\sum_{(x,y)\in\mathcal E_+}
\left|
a_{xy}^{1/2}u(x)-a_{xy}^{-1/2}u(y)
\right|^2
+
W(0)|u(0)|^2
\\
&\qquad\quad+
\sum_{x\neq0}
\left(
W(x)-\frac1{4|x|^2}
\right)|u(x)|^2.
\end{aligned}
\label{eq:main-remainder-decomposition}
\end{equation}
By Proposition~\ref{prop:weight-bound},
$
W(0)=\frac{12}{13}
$
and, for $x\neq0$,
\[
W(x)-\frac1{4|x|^2}
\geq
\rho(|x|^2).
\]
Dropping the nonnegative edge-square term in
\eqref{eq:main-remainder-decomposition} proves
\eqref{eq:explicit-remainder}.  Since $\rho(n)>0$ for every integer
$n\geq1$, the right-hand side of \eqref{eq:explicit-remainder} is strictly
positive whenever $u\not\equiv0$.
\end{proof}

\begin{proof}[Proof of Corollary~\ref{cor:sharp-constant}]
Theorem~\ref{thm:main} immediately gives \eqref{eq:main-hardy}, and the
inequality is strict for every nonzero finitely supported function.  In
particular, applying it to functions in $\mathcal D_0(\Z^3)$ gives
$\lambda_3\geq1/4$.

Proposition~\ref{prop:sharpness} gives a sequence of admissible functions
whose Rayleigh quotients tend to $1/4$.  Hence $\lambda_3\leq1/4$, and
therefore
\[
\lambda_3=\frac14.
\]
Thus the coefficient $1/4$ is sharp and is not attained by a nonzero
finitely supported function.  
\end{proof}

\begin{remark}[Further problems]
It remains to determine exact pure-power constants in fixed dimensions
$d\geq4$ and for nonlinear discrete $p$-energies.  It would also be useful
to classify all minimizing sequences for Corollary~\ref{cor:sharp-constant}
and to understand whether the edge-field construction has a natural
criticality interpretation on more general graphs.
\end{remark}

\appendix
\section{Motivation for the edge field}
\label{appendix:a_xy}

We explain the choice of the edge field in \eqref{eq:a-xy}.  The starting
point is the ground-state picture for the classical Hardy inequality.  In
dimension $d$, the positive radial function
\[
h_d(x)=|x|^{-(d-2)/2}
\]
satisfies
\[
-\Delta h_d
=
\frac{(d-2)^2}{4|x|^2}h_d
\qquad\text{on }\R^d\setminus\{0\}.
\]
Indeed, for a radial power $|x|^\alpha$ one has
\[
\Delta |x|^\alpha
=
\alpha(\alpha+d-2)|x|^{\alpha-2},
\]
and taking $\alpha=-(d-2)/2$ gives the identity above.  Thus $h_d$ is the
positive zero-energy solution for the critical Hardy operator
\[
-\Delta-\frac{(d-2)^2}{4|x|^2}.
\]
It is naturally viewed as a formal or generalized ground state, since it is
not itself an admissible extremizer for the Hardy inequality.  Ground-state
representations of this type are standard in the theory of Hardy
inequalities and critical Schr\"odinger operators; see
\cite{FrankSeiringer2008,PinchoverTintarev2006}.  In the discrete setting,
related ground-state transforms and Hardy weights on graphs are discussed in
\cite{KellerPinchoverPogorzelski2018}.

In dimension $d=3$, this becomes
\[
h(x)=|x|^{-1/2},
\qquad
-\Delta h
=
\frac{1}{4|x|^2}h.
\]
For a nearest-neighbor edge $x\sim y$ with $x,y\neq0$, write
\[
n=|x|^2,
\qquad
m=|y|^2,
\qquad
\delta=m-n.
\]
Then
\[
\frac{h(y)}{h(x)}
=
\left(\frac{n}{m}\right)^{1/4}
=
\left(1+\frac{\delta}{n}\right)^{-1/4}.
\]
Since $\delta=O(|x|)=O(n^{1/2})$ along nearest-neighbor edges,
\begin{equation}
\frac{h(y)}{h(x)}
=
1-\frac{\delta}{4n}
+
O\left(\frac1n\right),
\qquad |x|\to\infty.
\label{eq:appendix-ground-ratio}
\end{equation}

We therefore look for an edge factor with the same first-order behavior
and with the exact reciprocal property
\[
a_{yx}=a_{xy}^{-1}.
\]
A convenient way to impose reciprocity is to use a Cayley-type expression.
This leads to the one-parameter family
\begin{equation}
a_{xy}^{(c)}
=
\frac{n+m+c-\frac14(m-n)}
     {n+m+c+\frac14(m-n)}
=
\frac{5n+3m+4c}
     {3n+5m+4c},
\qquad c\geq0.
\label{eq:appendix-family}
\end{equation}
The numerator and denominator are positive on every nearest-neighbor edge.
Moreover, interchanging $x$ and $y$ leaves $n+m+c$ unchanged and reverses
the sign of $m-n$, so
\begin{equation}
a_{yx}^{(c)}
=
\bigl(a_{xy}^{(c)}\bigr)^{-1}.
\label{eq:appendix-reciprocity-family}
\end{equation}

The parameter $c$ affects only lower-order terms.  Indeed, if
$S_c=n+m+c$, then
\[
\begin{aligned}
a_{xy}^{(c)}
&=
\frac{S_c-\delta/4}{S_c+\delta/4}
=
1-\frac{\delta}{2S_c}
+
O\left(\frac{\delta^2}{S_c^2}\right)
=
1-\frac{m-n}{4n}
+
O\left(\frac1n\right),
\end{aligned}
\]
uniformly along nearest-neighbor edges as $|x|\to\infty$.  Thus every fixed
$c\geq0$ has the same first-order behavior as the ground-state ratio in
\eqref{eq:appendix-ground-ratio}.  The asymptotic matching determines the
coefficient $1/4$, but not the lower-order parameter $c$.

To choose $c$, consider the associated vertex weight
\[
W_c(x)
=
\sum_{y\sim x}\bigl(1-a_{xy}^{(c)}\bigr).
\]
Repeating the pairing and simplex-concavity argument from
Proposition~\ref{prop:weight-bound}, with
\[
A_c=8n+5+4c,
\]
gives the radial lower bound
\begin{equation}
W_c(x)
\geq
B_c(n)
:=
\frac{4(A_c-20n)}{A_c^2-100n}
+
\frac8{A_c}.
\label{eq:appendix-axis-family}
\end{equation}
Expanding this expression for large $n$ gives
\begin{equation}
B_c(n)-\frac1{4n}
=
\frac{32c-35}{64n^2}
+
O\left(\frac1{n^3}\right).
\label{eq:appendix-family-remainder-asymptotic}
\end{equation}
Thus $c>35/32$ is necessary for this lower bound to have a positive
inverse-fourth leading correction.  This asymptotic condition alone does
not guarantee positivity for every finite $n$.

We choose the simple value $c=2$.  In this case the exact difference in
\eqref{eq:appendix-axis-family} is
\[
B_2(n)-\frac1{4n}
=
\frac{928n^2+5356n-2197}
{4n(8n+13)(64n^2+108n+169)}
=
\rho(n),
\]
which is positive for every integer $n\geq1$ by
Proposition~\ref{prop:weight-bound}.  Thus $c=2$ gives the global positive
remainder used in the proof.  We do not claim that this choice is unique.

Finally, for $c=2$, \eqref{eq:appendix-family} becomes
\[
a_{xy}^{(2)}
=
\frac{5|x|^2+3|y|^2+8}
     {3|x|^2+5|y|^2+8},
\]
which is exactly the edge field $a_{xy}$ used in the paper.

The construction does not require these edge factors to arise from a positive function on the vertices.  If
\[
a_{xy}=\frac{H(y)}{H(x)}
\]
for some positive function $H$, then the product of the edge factors around every closed lattice path would necessarily equal one. 
We do not impose
this additional compatibility condition.  
Consider the closed path
\[
\begin{aligned}
x_0&=0,
&
x_1&=e_1,
&
x_2&=2e_1,
\\
x_3&=2e_1+e_2,
&
x_4&=e_1+e_2,
&
x_5&=e_2,
&
x_6&=x_0.
\end{aligned}
\]
For the field in \eqref{eq:a-xy}, direct substitution gives
\[
\begin{aligned}
\prod_{k=0}^{5}a_{x_kx_{k+1}}
&=
\frac{11}{13}
\frac{25}{31}
\frac{43}{45}
\frac{13}{11}
\frac{21}{19}
\frac{13}{11}
=
\frac{19565}{19437}
\neq1.
\end{aligned}
\]
Thus there is no positive function $H$ satisfying
$a_{xy}=H(y)/H(x)$ on every edge.  In this sense the field is genuinely
non-gradient.  The proof of the Hardy inequality does not require gradient
compatibility; it uses only positivity, reciprocity, and the global lower
bound for the associated vertex weight.

\end{document}